\documentclass[12pt,twoside]{amsart}
\usepackage{mathrsfs}
\usepackage{mathtools}
\usepackage{amssymb}
\usepackage{verbatim}
\usepackage{amsmath}
\usepackage{amsthm,thmtools,xcolor}
\usepackage[foot]{amsaddr}

\usepackage[colorlinks,linkcolor=blue,citecolor=blue, pdfstartview=FitH]
{hyperref}
\usepackage{bm}
\usepackage{a4wide}
\usepackage[latin1]{inputenc}
\usepackage[T1]{fontenc}
\usepackage{times}
\usepackage{hyperref}
\usepackage{amssymb,latexsym}
\usepackage{enumerate}

\newcommand{\dbar}{\ensuremath{\overline\partial}}

\usepackage[toc,page]{appendix}

\usepackage{etoolbox}
\makeatletter
\patchcmd\maketitle
  {\uppercasenonmath\shorttitle}
  {}
  {}{}
\patchcmd\maketitle
  {\@nx\MakeUppercase{\the\toks@}}
  {\the\toks@}
  {}
  {}{}
\patchcmd\@settitle
  {\uppercasenonmath\@title}
  {}
  {}{}
\patchcmd\@setauthors
  {\MakeUppercase{\authors}}
  {\authors}
  {}{}

\def\Im{{\rm Im}}

\makeatletter
\newcommand{\sumprime}{\if@display\sideset{}{'}\sum%
            \else\sum'\fi}
\makeatother

\begin{document}

\numberwithin{equation}{section}

% define theorem environments
\newtheorem{theorem}{Theorem}[section]
\newtheorem{proposition}[theorem]{Proposition}
\newtheorem{conjecture}[theorem]{Conjecture}
\def\theconjecture{\unskip}
\newtheorem{corollary}[theorem]{Corollary}
\newtheorem{lemma}[theorem]{Lemma}
\newtheorem{observation}[theorem]{Observation}
\newtheorem{definition}{Definition}
\numberwithin{definition}{section} %\def\thedefinition{\unskip}
\newtheorem{remark}{Remark}
\def\theremark{\unskip}
\newtheorem{kl}{Key Lemma}
\def\thekl{\unskip}
\newtheorem{question}{Question}
\def\thequestion{\unskip}
\newtheorem{example}{Example}
\def\theexample{\unskip}
\newtheorem{problem}{Problem}

\address{$^1$Department of Mathematical Sciences, Norwegian University of Science and Technology, NO-7491 Trondheim, Norway}
\email{johannes.testorf@ntnu.no}

\title[Geometry of the Time Frequency Torus]{A Complex Geometric Approach to the Discrete Gabor Transform and Discrete Localization Operators}


\author{{Johannes Testorf$^{1}$}}
\date{\today}

\begin{abstract} In a recent paper, the discrete Gabor transform was connected to a Gabor transform with a time frequency domain given by the flat torus. We show that the corresponding Bargmann-Fock spaces can be expressed as theta functions (or equivalently line bundles on Abelian varieties). We give applications of this viewpoint to frame results for the discrete Gabor transform. In particular, we get necessary conditions which hold in higher dimensions and can expand the known results in the one dimensional case, the primary tool being the theorem of the square. We also give an application to asymptotics of restriction operators which arises via the asymptotic behavior of Bergman kernels and Toeplitz operators for high tensor powers of line bundles and find that discrete time frequency restriction operators will exhibit "plunge" behaviors similar to those of time frequency restriction operators in other contexts.
\end{abstract}
\thanks{I am grateful to my advisors Franz Luef and Xu Wang for their continued support and for their comments on earlier versions of this paper. I would also like to thank Lu\'is Daniel Abreu for comments on an earlier version of this paper.}

\maketitle

\section{Introduction}
One of the central objects in the field of time frequency analysis is the short-time Fourier transform (STFT), or Gabor transform. The STFT is defined with respect to some window function $g\in L^2(\mathbb R^d),$ and is given by
$$
\mathbf{V}_gf(x,\xi) := \int\limits_{\mathbb R^d}f(t)\overline{g(t-x)}e^{-2\pi i \xi^T t}dt,
$$
for a function $f$ on $\mathbb R^d$. The idea behind the STFT is that one localizes the standard Fourier transform in time by multiplying by the window function. Then, by parametrizing the position of the window in time, we obtain a second "time variable" in the argument of the resulting transform. When the window function $g$ is given by a Gaussian, a case to which we will give special attention to in this article, we will also refer to the STFT as the Gabor transform.

Now if we introduce the time, modulation, and time-frequency shift operators, $\mathbf{T}_xf(t):=f(t-x), \mathbf{M}_\xi f(t):=e^{2\pi i \xi^T t}f(t)$, and $\pi (x,\xi):= \mathbf{M}_\xi \mathbf{T}_x$ respectively, we may rewrite the STFT as 
$$
\mathbf{V}_gf(x,\xi) = \langle f, \pi (x,\xi) g\rangle.
$$
Since the STFT is the Fourier transform  of a localized piece of $f$, it may be discretized in a way similar to the standard Fourier transform. To define the discrete Gabor transform (DGT), we start by selecting $d$ integers $N_1,...,N_d$, representing the length of time domain in each dimension, and setting $\mathfrak{N}:=\mathrm{diag}(N_1,...,N_d).$ Then the DGT is given by
\begin{align*}
\mathbf{V}_\mathbf{g}\mathbf{f}[k,l] = \sum\limits_{m\in I_\mathfrak{N}} \mathbf f [m] \overline{\mathbf g [m-k]}e^{-2\pi i l^T \mathfrak{N}^{-1}m},\quad \mathbf{f,g}\in \mathbb C^{N_1 \times ... \times N_d},
\end{align*}
where $I_\mathfrak{N}:=\mathbb Z_{N_1} \times ... \times \mathbb Z_{N_d},$ and $k,l\in I_\mathfrak{N}$. Though the results of this paper can certainly be generalized to a general $\mathfrak{N}$, for the sake of legibility we will restrict ourselves to the case where $N_1=...=N_d=N$. We then also use the notation $I_N=(\mathbb Z_N)^d$.

In \cite{FlatTori}, the study of the DGT was related to the following setting: consider the Feichtinger algebra defined as
\[
{S}_0(\mathbb R^d):=\left\{f\in L^2(\mathbb R^d): \mathbf{V}_{h_0}f\in L^1(\mathbb R^{2d})\right\},
\]
where $h_0:=e^{-\pi |z|^2}$. Now consider for $f\in \mathcal{S}_0(\mathbb R^d)$ the double periodization operator,
$$
\mathbf{\Sigma}_{N}f:=\sum\limits_{k_1,k_2\in \mathbb Z^d}\mathbf{M}_{k_2}\mathbf{T}_{{N}k_1}f.
$$
By \cite{FlatTori} this operator has the following properties: its image lies in $S_{N},$ the space spanned by Dirac combs of the form
$$
\epsilon_n:=\sum\limits_{k\in \mathbb Z^d} \delta_{n+{N}k},\quad n\in I_N
$$
and has a Hilbert space structure given by
$$
\langle\epsilon_n,\epsilon_m\rangle_{S_{N}}:=\delta_{m}^n,
$$
where $\delta_{m}^n$ denotes the Kronecker delta. Furthermore, the STFT of $\mathbf{\Sigma}_{N}f$ is quasiperiodic in the sense,
$$
\mathbf{V}_g\mathbf{\Sigma}_{N}f(x+{N}k,\xi+l) = e^{-2\pi i k^T{N}\xi}\mathbf{V}_g\mathbf{\Sigma}_{N}f(x,\xi), \quad f,g\in S_0(\mathbb R^d), k,l\in \mathbb Z^d, 
$$
and thus can be said to have the torus $\mathbb T_{N}:=\mathbb R^{2d}/({N}\mathbb Z^d\times \mathbb Z^d)$ as its time-frequency domain. The connection to the discrete setting is the following result from \cite{FlatTori} which can be directly generalized into the higher dimensional case.
\begin{theorem}[{\cite[Theorem 1.1]{FlatTori}}]\label{thm:DGTtoSN}
    Consider the periodization and sampling operator 
    $$
    \mathbf P^N_{samp} f := \left(\sum\limits_{k\in \mathbb Z^d} f(n - kN)\right)_{n\in I_N}, \quad f\in \mathcal{S}_0(\mathbb R^d).
    $$
    Now let $f,g\in S_0(\mathbb R^d)$, and $\mathbf{f}_N:=\mathbf{P}^N_{samp}f$ and the same for $\mathbf{g}_N$. Then we have for $k,l \in I_N$,
    $$
    \mathbf{V}_g(\mathbf{\Sigma}_Nf)(k,\tfrac{l}{N}) = \mathbf{V}_{\mathbf{g}_N}\mathbf{f}_N [k,l].
    $$
\end{theorem}
We remark that in less precise terms, this result can be seen as reflective of the fact that discrete Fourier transforms have a periodic image. Thus it makes sense to study the DGT via functions on the torus.
\subsection{Frames for the DGT}
The first objective of this article is to extend \cite[Theorem 1.3]{FlatTori} to the case of higher dimensions. This concerns when a set $\{(k_j,l_j)\}_{j=1}^{K}\subseteq I_N^2$ generates a frame for the DGT, by which we mean that for given $\mathbf{g}\in \mathbb C^{N^d}$ we may find constants $A,B>0$ such that for any $\mathbf{f}\in \mathbb C^{N^d}$
\[
A\|\mathbf f\|^2\le \sum\limits_{j=1}^K|\mathbf{V}_{\mathbf{g}}\mathbf f[k_j,l_j]|^2 \le B\|\mathbf f\|^2.
\]
Since we look at norms on a finite dimensional vector space, we can in fact choose any norm on $\mathbb C^{N^d}$. The significance of $\{(k_j,l_j)\}_{j=1}^{K}$ generating a frame is that it means that one may recover $\mathbf{f}$ from the values of $\mathbf{V}_{\mathbf{g}}\mathbf{f}$ on $\{(k_j,l_j)\}_{j=1}^{K}$. (See \cite{FTFA} for both a general reference on frame theory and the study of Gabor frames in the continuous case.) We will study frames when the window vector $\mathbf{g}$ is related to a Gaussian function of the kind constructed below. 
\\~~\\
Consider the Siegel upper half space 
$$
\mathfrak{H}:=\{\Omega\in \mathfrak{gl}(d,\mathbb C): \Omega = \Omega^T,\, \mathrm{Im}\,\Omega \text{ is positive definite}\}.
$$
Now for any $\Omega\in \mathfrak{H}$ and $N\in \mathbb N$, we define  the corresponding Gaussian function as
$$
h_0^{\Omega/N }(t):=\overline{e^{\pi i t^T ({\Omega}N^{-1})t}},\quad t\in \mathbb R^d,
$$
and the corresponding window for the DGT which we define as
\[
\mathbf h_0^{\Omega/N}:=\mathbf P^N_{samp}(h^{\Omega/N}_0).
\]
Now we may state our first main result which is a necessary frame criterion.
\\~~\\
\textbf{Theorem A.} \emph{The set $\{(k_j,l_j)\}_{j=1}^{K}\subseteq I_N^2$ does not induce a frame for the DGT with window $\mathbf h_0^{\Omega/N}$ if the following holds. For the points
\[
z_j:= -i(N^{-1} \Omega k_j+\frac{l_j}{N}),
\]
there exist further points $v_1,...,v_N\in \mathbb C^d$ with
\[
\sum_{j=1}^N v_j \in \Lambda_\Omega
\]
where
\[
\Lambda_\Omega:=-i\Omega\mathbb Z^d \oplus i\mathbb Z^d,
\]
such that
\[
\{z_1,...,z_K\}\subseteq\bigcup_{j=1}^Nt_{v_j}(\{z\in\mathbb C^d|\vartheta_1(iz,\Omega)=0\})
\]
where $\vartheta_1$ is the higher dimensional Riemann theta function
    $$
    \vartheta_1(z,\Omega) = \sum\limits_{k\in \mathbb Z^d}\exp(\pi ik^T\Omega k+ 2\pi i k^Tz)
    $$
and $t_{v_j}(z):=z-v_j$.
}
\\~~\\
This result is intimately related to the geometry of the complex abelian variety $\mathbb T_\Omega:=\mathbb C^d/\Lambda_\Omega$. Indeed, to prove Theorem A we will make use of a Bargmann transform whose image will lie in the theta functions associated to a polarization of $\mathbb T_\Omega$.

Theorem A is a partial extension of \cite[Theorem 1.3]{FlatTori} to general dimension. To our knowledge there are no other criteria for discrete Gaussian Gabor frames in higher dimension. In case $d=1$, our methods also give the following full characterization.
\begin{theorem}\label{thm:1dCase}
     For $d=1$ the points $\{(k_j,l_j)\}_{j=1}^{K}\subseteq I_N^2$ induce a Gabor frame for the DGT on $\mathbb C^{N}$ with window $\mathbf h^{\Omega/N}_0$, if $K>N$. We do not get a frame for $K<N$. If $K=N$ we get a frame precisely when the following is \textbf{not} fulfilled: $N$ is even and  
    $$\sum\limits_{j=1}^N(k_j,l_j)\in N\mathbb Z \times N\mathbb Z.$$
\end{theorem}
When $\mathrm{Re\,}\Omega = 0$ this is precisely \cite[Theorem 1.3]{FlatTori}. Thus our result here also gives a slight generalization in the one dimensional case to more general complex dilations of the window. However, it is likely possible to use the methods there to gain this generalization as well. 

Moreover, we highlight that the proof in this article relies more on the algebraic properties of theta functions rather than the approach in \cite{FlatTori} which relies heavily on Cauchy's argument principle. This difference in approach is what allows us to obtain results in higher dimension. However, to obtain the full characterization in dimension 1 we make use of Rouch\'e's theorem, meaning that we have implicitly still used methods involving 1 dimensional contour integrals. Thus it is unclear how one might prove the converse direction to Theorem A.

\subsection{Asymptotics for Localization Operators}
Another set of important objects in the field of time frequency analysis are the localization operators. Here we will study discrete localization operators and their asymptotic properties as $N\to \infty$. 

While in reality, these may be defined with respect to a matrix, in order to speak of asymptotic properties, we will need to sample a function on the continuous time--frequency domain which in our case we take to be $[0,1)^{2d}$.  For a function $a$ on $[0,1)^{2d}$, we define the associated localization operator for the window $\mathbf{h}^{\Omega/N}_0$ (assuming this window does not vanish) as
\begin{align}
    \mathbf{R}^{N}_{a}\mathbf{f}:=\frac{1}{N^d|\mathbf{h}^{\Omega/N}_0|^2}\sum\limits_{
    k,l\in I_N}a\left(\frac{k}{N},\frac{l}{N}\right)\mathbf{V}_{\mathbf{h}^{\Omega/N}_0}\mathbf{f}[k,l]\pi[k,l]\mathbf{h}^{\Omega/N}_0,\quad \mathbf{f}\in \mathbb C^{N^d}.
\end{align}
Here, $\pi[k,l]$ is the discrete time-frequency shift
\[
\pi[k,l]\mathbf{g}[m]:=\mathbf{g}[m-k]e^{2\pi N^{-1} il^Tm}.
\]
 The interpretation of the localization is this: first we transform to a function on the time frequency domain (in this case we interpret this as the torus), then we multiply by some measurable function and transform back to get a new function on the time domain. We do not focus heavily on the inversion of the DGT since this will introduce numerical considerations which are beyond the scope of this article. Instead we refer the reader to \cite{PLS21} for more on this. 

    Localization operators have been studied significantly in time frequency analysis (see \cite{Daub,Lindholm,CG03,NT22} just to name a few.) 

    We concern ourselves with the spectrum of $\mathbf{R}^{N}_{a}$ in the large $N$ limit. Our main result is the following, which is essentially the results of \cite{BerToep} adapted to our situation.
    \\~~\\
    \textbf{Theorem B. }\emph{Let $a$ be real valued, Riemann integrable function on $[0,1)^{2d}$. Consider the function $N_{<\alpha}(\mathbf{R}^{N}_a)$ which is the number of eigenvalues of $\mathbf{R}^{N}_a$ less than $\alpha$ counting multiplicities, then for all but possibly countably many $\alpha$,
        $$\lim\limits_{N\to \infty}N^{-d} N_{<\alpha}(\mathbf{R}^{N}_a) = {\mathrm{Vol}(a<\alpha)}.$$ 
        with a similar bound for $N_{>\alpha}(\mathbf R^{N}_a).$ Here $\mathrm{Vol}$ denotes the Lebesgue volume on $[0,1]^{2d}$.}
        \\~~\\
We stress that this result follows from the asymptotic expansion of the Bergman kernel which is due independently to Catlin, and Zelditch \cite{Catlin, Zelditch98}. There are essentially two parts to our proof. The first is to describe asymptotics of localization operators for the space $S_N$, which follows almost directly from the asymptotic expansion on the diagonal. In particular, it is important that the leading term of this expansion is equal to 1. The second step is to compare the discrete localization operators to the ones for $S_N$. Here the off diagonal expansion of the Bergman kernel plays a key role.

Indeed, most of the machinery behind Theorem B is by no means new. As was mentioned our approach is essentially just to relate our situation to the setup of \cite{BerToep}. (One may find similar results for the specific case of the torus in \cite{Zeldtich10}). Moreover, one may find an approach to our results without complex analysis in \cite{CoherentStates}. What is novel as far as we know is the application to discrete localization operators as well as the intermediate results for localization operators on $S_N$ which we can obtain easily due to our connection to the study of theta functions. What is also new as far as we know is the link between localization operators on $S_N$ and discrete localization operators.

Theorem B also implies that the localization operators in this setting develop "plunge" regions like the ones in \cite{Daub,FeiNow}. By this we mean the following phenomenon: if we choose our symbol to be the indicator function of a set with "nice" boundary conditions, then the spectrum is contained in $[0,1]$. Moreover, if we order the eigenvalues from largest to smallest, the first eigenvalues are all very close to one, until they very suddenly rapidly decay (this is what is meant by plunge) and are then all very close to zero. Note however, that these statements hold in an asymptotic sense and only for almost all $\alpha$, thus our theorem is insufficient to describe the spectrum for any specific fixed $N$.

If we set the function $a$ to be an indicator function, we can clearly see the same behavior taking place, as asymptotically, almost all of the eigenvalues are approaching either zero or one. 

\section{Preliminaries on Theta Functions}
Most of our results in this paper follow from the theory of theta functions. Our exposition will essentially follow the results in \cite{Bea} (see also \cite{BirkenhakeLange, TataTheta1} for more background). The theory of theta functions is intimately related to algebraic geometry on complex tori. In particular, much of what we will say here is more appropriately formulated in the language of complex line bundles, but since this would risk alienating much of this article's potential readership, we will formulate everything in the language of functions on $\mathbb C^n$, making a few remarks about the line bundle formulation along the way. 

We start by defining what we mean by theta function.
\begin{definition}
    Consider a lattice $\Lambda\cong \mathbb Z^{2d}$ contained in $\mathbb C^d$ and a family of invertible holomorphic functions $(e_{\lambda})_{\lambda\in \Lambda}$ such that
\begin{align}\label{eq:cocycle}
e_{\lambda_1+\lambda_2} = e_{\lambda_1}(z+\lambda_2)e_{\lambda_2}(z),\quad \lambda_1,\lambda_2\in \Lambda,z \in \mathbb C^d.
\end{align}
Then an entire function $\theta\in \mathcal{O}(\mathbb C^d)$ is called a theta function if it fulfills the quasi-periodicity condition,
\begin{align}
\theta(z +\lambda) = e_{\lambda}(z)\theta(z), \quad \lambda\in \Lambda, z \in \mathbb C^d.
\end{align}
We call the family $(e_\lambda)_{\lambda\in \Lambda}$ the associated family of multipliers. We denote the space of theta functions associated to a family of multipliers as $\Theta(e_\lambda)$. We consider two families of multipliers $(e_\lambda)_{\lambda\in \Lambda},(\tilde e_\lambda)_{\lambda\in \Lambda} $ to be equivalent if there exists a nonvanishing holomorphic function $h:\mathbb C^n\to \mathbb C^*$ with 
\[
e_\lambda(z) = \frac{h(z+\lambda)}{h(z)}\tilde e_\lambda(z)
\]
\end{definition}
We remark that the equivalence classes we have defined on the systems of multipliers mean that each class of multipliers defines a unique line bundle on the torus $\mathbb C^d/\Lambda$. Thus classes of theta functions for equivalent systems of multipliers will induce equivalent geometric properties. In particular, they will share systems of zero sets for their associated theta functions. (This last fact is important and should also be obvious from our definition of equivalence.) We also have the converse of this property.

\begin{proposition}\label{prop:EquiConverse}
    If the functions $\theta\in\Theta(e_\lambda)$ and $\tilde \theta\in \Theta(\tilde e_\lambda)$ are not identically equal to 0 and have the same vanishing sets, counting multiplicities, then $e_\lambda$ is equivalent to $\tilde e_\lambda$.
\end{proposition}
\begin{proof}
    Since the two theta functions we have defined have the same vanishing, we may extend ${\theta}/{\tilde \theta}$ to a global non-vanishing holomorphic function. We have trivially that
    \[
    \theta = ({\theta}/{\tilde \theta})\tilde \theta.
    \]
    Thus it follows that
    \[
    \theta(z+\lambda) = e_\lambda(z)({\theta}(z)/{\tilde \theta}(z))\tilde \theta(z)=(\theta(z+\lambda)/\tilde\theta(z+\lambda))\tilde e_\lambda(z)\tilde \theta (z).
    \]
    This in turn, means that
    \[
    e_\lambda(z) = \frac{\theta(z+\lambda)/\tilde \theta(z+\lambda)}{\theta(z)/\tilde \theta(z)}\tilde e_\lambda(z).
    \]
\end{proof}

We will consider theta functions which are much more special than the general definition we have just given. Consider a sesquilinear, conjugate symmetric form $H$ on $\mathbb C^d$ and such that the skew symmetric form 
\[
c_1:=\mathrm{Im}\,H
\]
takes integral values on $\Lambda$. Consider furthermore a semicharacter of $\Lambda$, i.e. a function $\chi:\Lambda\to S^1$ such that
\[
\chi(\lambda_1+\lambda_2)=\chi(\lambda_1)\chi(\lambda_2)(-1)^{c_1(\lambda_1,\lambda_2)}, \quad \lambda_1,\lambda_2\in \Lambda.
\]
Now for this data we define the multipliers 
\[
e_\lambda(H,\chi)(z):=\chi(\lambda)e^{\pi(H(\lambda,z)+\frac{1}{2}H(\lambda,\lambda))}.
\]
We call such multipliers a canonical system of multipliers. Now we have the following definition.
\begin{definition}
The admissible pairs $(H,\chi)$ admit a group structure via the operation 
\[
(H_1,\chi_1)\cdot (H_2,\chi_2):=(H_1+H_2,\chi_1\chi_2).
\]We call this group the Picard group of the torus $\mathbb C^d/\Lambda$ and we denote it by $\mathrm{Pic}(\mathbb C^d/\Lambda)$. We denote the subgroup of those elements with $H=0$ by $\mathrm{Pic}^\circ(\mathbb C^d/\Lambda)$.
\end{definition}
    
We will also need the well known Appell-Humbert Theorem. (See e.g. \cite[Theorem 2.6]{Bea} for proof)
\begin{theorem}[Appell-Humbert Theorem]
    Any system of multipliers $(e_\lambda)_{\lambda\in \Lambda}$ is equivalent to a unique pair $(H,\chi)$ in the Picard group.
\end{theorem}
We remark that this means that in the language of line bundles, this makes our Picard group isomorphic to the group of complex line bundles on the torus with the operation of tensor multiplication, which is what is usually meant by Picard group. We also remark that the form $c_1$ is in fact the first Chern class of the associated line bundle. This fact will play a small role later on. This also points to further intuition about the pair $(H,\chi).$ Namely, that $H$ contains the topological information about the induced line bundle (since $c_1$ is actually a topological invariant). Indeed, the line bundles corresponding to $H=0$ are precisely those which are topologically trivial. So $H$ contains topological information while $\chi$ is encoding purely geometric information.

It will be very important to understand how the canonical multipliers behave under translation. We have the following lemma from \cite[Lemma 2.7]{Bea}
\begin{lemma}\label{le:ThetaTranslate}
    We have that for $v\in \mathbb C^d$, $e_\lambda(H,\chi)(z+v)$  is equivalent to $e_\lambda(H,\chi \cdot e^{2\pi ic_1(\lambda,v)})(z)$.
\end{lemma}
\begin{proof}
    By direct computation, we have that
    \[
    e_\lambda(H,\chi)(z+v) = e_\lambda(H,\chi)(z)e^{\pi H(\lambda,v)}.
    \]
    Define the holomorphic function $h_v(z):=e^{-\pi H(v,z)}$. Then we have that
    \[
    \frac{h_v(z+\lambda)}{h_v(z)} = e^{-\pi H(v,\lambda)}
    \]
    and furthermore,
    \[
    \frac{h_v(z+\lambda)}{h_v(z)}e_\lambda(H,\chi)(z+v) = e_\lambda(H,\chi)(z)e^{-\pi( H(v,\lambda)-H(\lambda,v))}
    \]
    which implies the lemma after simplifying.
\end{proof}
This lemma will have an important consequence once the following definition is in place.
\begin{definition}
    Keeping the notations and definitions we have put in place thus far, we define the symplectic dual subgroup (or lattice when $c_1$ is nondegenerate) to $\Lambda$ with respect to $c_1$ as
    \[
    \Lambda^\circ :=\left\{z\in \mathbb C^d:c_1(z,\lambda)\in \mathbb Z \text{ for all $\lambda\in \Lambda$} \right\}.
    \]
    Then the quotient $\Lambda^\circ /\Lambda$ defines a subgroup of the torus $\mathbb C^d/\Lambda$ which we call the translational stabilizer.
\end{definition}
Now we may formulate the well known theorem of the square.
\begin{theorem}[Theorem of the square]
    Let $(e_\lambda)_{\lambda\in \Lambda}$ be a system of multipliers. Then the map 
    \[
    \tau:\mathbb C^d/\Lambda\to \mathrm{Pic}^\circ (\mathbb C^d/\Lambda), \quad v+\Lambda\mapsto (0, e^{2\pi ic_1(\cdot,v)})
    \]
    is a group homomorphism with kernel equal to the translational stabilizer $\Lambda^\circ/\Lambda$.
\end{theorem}
In our formulation of the theorem, there is almost nothing to prove. However, in view of Lemma \ref{le:ThetaTranslate}, this has a much deeper meaning. Namely, if we consider $(H,\chi)$ to be the pair representing $(e_\lambda)_{\lambda\in \Lambda}$, then 
$$
(0, e^{2\pi ic_1(\cdot,v)})=(-H+H,\chi^{-1}\chi e^{2\pi ic_1(\cdot,v)})
$$
is the pair representing the family of multipliers
\[
(e_\lambda^{-1}\cdot e_\lambda(\cdot +v ))_{\lambda\in \Lambda}.
\]
Thus, the morphism from the theorem of the square can also be understood as the map
\[
v+\Lambda\mapsto (e_\lambda^{-1}\cdot e_\lambda(\cdot +v ))_{\lambda\in \Lambda}.
\]
The most important consequence of this is in the case of a pair of the form $(NH, \chi^N)$ for a number $N\in \mathbb N$. If we now pick a set of $N$ points $v_1,...,v_N\in \mathbb C^d/\Lambda$ such that the sum $\sum_{i=1}^Nv_i$ lies in the translational stabilizer $\Lambda^\circ /\Lambda$ for the pair $(H,\chi)$, then since the map $\tau$ is a homomorphism, we get that the system of multipliers
\[
(e_\lambda^{-N}\cdot e_\lambda(\cdot +v_1)\cdot ... \cdot e_\lambda (\cdot +v_N))_{\lambda\in \Lambda}
\]
is equivalent to the trivial multiplier system. This implies in particular that $(e_\lambda(\cdot +v_1)\cdot ... \cdot e_\lambda (\cdot +v_N))_{\lambda\in \Lambda}$ is equivalent to $(e_\lambda^{N})_{\lambda\in \Lambda}$. Since the systems of vanishing sets of these are thereby equal, we have the following result.
\begin{lemma}\label{le:ThetaZeroes}
    Let $(e_\lambda)_{\lambda\in \Lambda}$ be a system of multipliers. Furthermore, let $A$ be the vanishing set of some theta function in $\Theta(e_\lambda)$. Then for any set of points $v_1,..,v_N$ whose sum $\sum_{i=1}^Nv_i$ lies in the translational stabilizer $\Lambda^\circ /\Lambda$ associated to $(e_\lambda)_{\lambda\in \Lambda}$, there exists a $\theta\in \Theta(e_\lambda^N)$ whose vanishing set is given by
    \[
    \{\theta=0\} = \bigcup_{i=1}^Nt_{v_i}(A),
    \]
    where $t_{v_i}$ is the translation mapping $z\mapsto z-v_i$.
\end{lemma}
\begin{proof}
    It is obvious that if $\vartheta$ is an element of $\Theta(e_\lambda)$, then $\prod_{i=1}^N\vartheta(\cdot +v_i)$ is an element of $\Theta(e_\lambda(\cdot +v_1)\cdot ... \cdot e_\lambda (\cdot +v_N))$ with the vanishing set $\bigcup_{i=1}^Nt_{v_i}(A)$. But since this has a multiplier equivalent to $e^N_\lambda$, there is an element in $\Theta(e_\lambda^N)$ with the same vanishing set.
\end{proof}

We now construct the special type of lattice we will consider. Recall from the introduction that we had defined the Siegel upper half space as
$$
\mathfrak{H}:=\{\Omega\in \mathfrak{gl}(d,\mathbb C): \Omega = \Omega^T,\, \mathrm{Im}\,\Omega \text{ is positive definite}\}.
$$
Now consider $\Omega\in \mathfrak{H}$. We define the lattice
\[
\Lambda:=-i\Omega\mathbb Z^d \oplus i\mathbb Z^d.
\]
We remark that this is rotated via multiplication with $-i$ from what is typically found in the literature. Indeed, this convention might simplify things if one is purely interested in the study of theta functions, however, we consider this convention since it is compatible with the transforms we will consider later on. 

In algebro-geometric literature, the matrix $\Omega$ is called the period matrix.

We also consider a special $H$. Let $N\in \mathbb N$ be a natural number. Then we set
\[
H:=N(\operatorname{Im}\Omega)^{-1}
\]
This setup is special and in this case, $H$ is called a polarization of the torus $\mathbb C^d/\Lambda$. Now we have by \cite{Bea} or \cite[Theorem 3.3.3]{BirkenhakeLange} the following theorem.
\begin{theorem}\label{thm:ThetaDim}
    For the $H$ as defined above, we have that 
    \[
    \dim(\Theta(e_\lambda(H,\chi)))=N^d.
    \]
\end{theorem}
\begin{proof}
    \cite[Theorem 3.3.3]{BirkenhakeLange} tells us that $\dim(\Theta(e_\lambda(H,\chi)))$ is given by the reduced Pfaffian of $c_1$. But this is precisely $N^d$ (in \cite{BirkenhakeLange} this is essentially by definition). 
\end{proof}
Lastly we turn to analysis of theta functions in this special setup. We first seek a function $\phi$ such that for any $\theta\in \Theta(H,\chi)$ we have that $|\theta|^2e^{-\phi}$ is invariant under the action of the lattice $\Lambda$, i.e. it defines a function on the torus $\mathbb C^d/\Lambda$. A straightforward computation shows that this is fulfilled by
\[
\phi(z):=\pi H(z,z) = \pi N \bar z^T(\operatorname{Im}\Omega)^{-1}z.
\]
We remark that this should also clarify why $c_1$ as we defined it earlier is precisely the first Chern class of the associated line bundle. It is defined as the class of the curvature form of the metric $\phi$ which is given by
\[
\frac{i}{2\pi}\partial\dbar \phi = c_1.
\]
Since this is a positive definite form this defines a natural K\"ahler structure on the torus, and in particular, it gives rise to a natural volume form
\[
dV_{\mathbb T,N}:=\frac{i^{d}}{2^d}N^d\det(\operatorname{Im}\Omega)^{-1}dz_1\wedge d\overline{z}_1\wedge \dots \wedge dz_d\wedge d\overline{z}_d.
\]
As it will turn out, this will also be the natural volume form to consider when investigating the Gabor transform. Thus we will equip the space $\Theta(e_\lambda(H,\chi))$ with the $L^2$ norm
\begin{align}
    \langle \theta_1,\theta_2\rangle_{L^2} = \int_{\mathbb C^d/\Lambda}\theta_1\overline{\theta_2}e^{-\phi}dV_{\mathbb T,N}, \quad \theta_1,\theta_2\in \Theta(e_\lambda(H,\chi)).
\end{align}
Since this is a finite dimensional Hilbert space, it possesses a reproducing kernel, which we call the Bergman kernel (since this is a Hilbert space of holomorphic functions). Denoting the Bergman kernel by $K$, as for any other reproducing kernel, we may write it in terms of an orthonormal basis $\{\theta_i\}$ as
\[
K(z,w):=\sum_{i}\theta_i(z)\overline{\theta_i}(w).
\]
Thus we have that the kernel is holomorphic in the first argument and antiholomorphic in the second. We will also be interested in the Bergman function which is given by
\[
B(z):=K(z,z)e^{-\phi}.
\]
By standard arguments about reproducing kernels, (see e.g. \cite{BerToep}) this has the extremal characterization
\begin{align}
B(z)=\sup_{\theta\in \Theta(e_\lambda(H,\chi))\setminus\{0\}}\frac{|\theta(z)|^2e^{-\phi(z)}}{\|\theta\|^2_{L^2}}.
\end{align}
\section{The Gabor Transform on \texorpdfstring{$S_N$}{SN}}
We now collect the results from \cite{FlatTori} that we will need. Note that we will change some of the conventions slightly. 

Recall first from the introduction the definition of the double periodization operator,
\begin{align*}
    \mathbf{\Sigma}_N f (t):=\sum\limits_{k_1,k_2\in \mathbb Z^d}f(t-k_1N)e^{2\pi i t^Tk_2 } = \sum\limits_{k_1,k_2\in \mathbb Z^d} \mathbf{M}_{k_2}\mathbf{T}_{k_1N}f.
\end{align*}
We have the following.
\begin{lemma}[{\cite[Lemma 4]{FlatTori}}] 
    The double periodization operator is a well defined operator $S_0(\mathbb R^d)\to S_0'(\mathbb R^d)$  with unconditional convergence in the weak-* topology. 
\end{lemma}
In order to express the double periodization operator better we also introduce the periodization operators of length $N$:
$$
\mathbf P^Nf(t):=\sum\limits_{k\in \mathbb Z^d} f(t - kN).
$$
Now by \cite{FlatTori} we have the formulas,
\begin{align}\label{eq:DoublePerSinglePer}
    \mathbf{\Sigma}_N f(t) = \sum\limits_{n\in I_N}\mathbf P^Nf(n)\epsilon_n,\quad f\in S_0(\mathbb R^d),\\
        \langle \mathbf{\Sigma}_Nf, g \rangle_{S_0'\times S_0}=\langle\mathbf{\Sigma}_Nf,\mathbf{\Sigma}_Ng\rangle_{S_N},\quad f,g\in S_0(\mathbb R^d).
\end{align}
We remark that the normalization we have chosen removes the factor $N$ which was obtained in \cite{FlatTori}. In particular, the double periodization operator on $S_0(\mathbb R^d)$ takes its image in $S_N$ the space defined in the introduction as 
the linear span of Dirac combs of the form
$$
\epsilon_n:=\sum\limits_{k\in \mathbb Z^d} \delta_{n+{N}k}.
$$
Moreover, we have that $\mathbf{\Sigma}_N$ maps surjectively into $S_N.$

The STFT of the double periodization operator can be expressed via the Zak transform, as the Zak transform is a time-frequency representation which has quasi-periodic functions as its image. The Zak transform with dilation parameter $N$ is defined as
\begin{align*}
    \mathbf{Z}_Nf(x,\xi):=\sum\limits_{k\in \mathbb Z^d}f(x-Nk)e^{2\pi i Nk^T\xi}.
\end{align*}
The Zak transform has the following properties which were used to show the results in \cite{FlatTori} (see \cite{FTFA}): it is quasi-periodic
\begin{align}
    \mathbf{Z}_N f(x,\xi+\tfrac{k}{N})=\mathbf{Z}_N f(x,\xi), \text { and } \mathbf{Z}_Nf(x+Nk,\xi) = e^{2\pi iN k^T\xi} \mathbf{Z}_Nf(x,\xi), \quad k\in\mathbb Z^d.
\end{align}
It has the following behavior under time-frequency shifts:
\begin{align}
    \mathbf{Z}_N(\pi(u,\eta)f)(x,\xi) =e^{2\pi i\eta^T x}\mathbf{Z}_N f(x-u,\xi-\eta),\quad u,\eta\in \mathbb R^d .
\end{align}
And lastly we have that
\begin{align}
    \int\limits_{[0,N^{-1}]^d}\int\limits_{[0,N]^d}\mathbf{Z}_Nf(x,\xi)\overline{\mathbf{Z}_Ng(x,\xi)}dxd\xi = N^{-d}\langle f,g \rangle_{L^2}.  
\end{align}
In particular, the formalism of the proofs of \cite{FlatTori} can be generalized to our situation thanks to these properties holding in general dimension $d$.

Now a computation from \cite{FlatTori} shows that the STFT of a basis element $\epsilon_n$ is given by
\begin{align}\label{eq:STFTeps}
    \mathbf{V}_g\epsilon_n(x,\xi)  & = e^{-2\pi i \xi^T n}\mathbf{Z}_N\overline g(n-x,\xi).
\end{align}
Linear extension in combination with (\ref{eq:DoublePerSinglePer}) gives
\[
\mathbf{V}_g\mathbf{\Sigma}_N f = \sum\limits_{n\in I_N}\mathbf{P}^Nf(n) e^{-2\pi i \xi^T n}\mathbf Z_N\bar g(n-x,\xi).
\]
Moreover, recall that we have by Theorem \ref{thm:DGTtoSN} that for
$\mathbf{f}_N:=\mathbf{P}^N_{samp}f$ and $\mathbf{g}_N:=\mathbf{P}^N_{samp}g$,
    $$
    \mathbf{V}_g(\mathbf{\Sigma}_Nf)(k,\tfrac{l}{N}) = \mathbf{V}_{\mathbf{g}_N}\mathbf{f}_N [k,l],\quad k,l \in I_N.
    $$
We turn to the $L^2$ properties of the STFT on $S_N$. Note that we have the quasi-periodicity 
$$
\mathbf{V}_g\varphi(x+{N}k,\xi+l) = e^{-2\pi i k^T{N}\xi}\mathbf{V}_g\varphi(x,\xi), \quad \varphi\in S_N,g\in S_0(\mathbb R^d), k,l\in \mathbb Z^d. 
$$
Thus for $g_1,g_2\in S_0(\mathbb R^d)$ and $\varphi_1,\varphi_2\in S_N$ we have that the function
\[
\mathbf{V}_{g_1}\varphi_1\overline{\mathbf{V}_{g_2}\varphi_2}
\]
is periodic with respect to the lattice $N\mathbb Z^d\oplus\mathbb Z^d$. Thus we define an $L^2$ norm by integrating only over the fundamental domain of the associated torus, i.e.
\[
\langle\mathbf{V}_{g_1}\varphi_1,\mathbf{V}_{g_2}\varphi_2\rangle_{L^2}:=
\int\limits_{[0,1]^d}\int\limits_{[0,N]^d}\mathbf{V}_{g_1}\varphi_1(x,\xi)\overline{\mathbf{V}_{g_2}\varphi_2}(x,\xi)dxd\xi.
\]
Now one of the most important results about the Gabor transform from \cite{FlatTori} is the Moyal orthogonality relation in this context.
\begin{theorem}[Moyal Orthogonality Relations]
    Let $g_1,g_2\in S_0(\mathbb R^d)$ and $\varphi_1,\varphi_2\in S_N$. Then we have that
    \[
    \langle\mathbf{V}_{g_1}\varphi_1,\mathbf{V}_{g_2}\varphi_2\rangle_{L^2}  =\langle\varphi_1,\varphi_2\rangle_{S_N}\langle g_2, g_1 \rangle_{L^2(\mathbb R^d)}.
\]
\end{theorem}
Here we mean by $\langle g_1, g_2 \rangle_{L^2(\mathbb R^d)}$ the unweighted $L^2$ inner product with respect to the Lebesgue volume element.

A standard consequence of the Moyal relations is the following inversion formula from \cite{FlatTori}. 
\begin{theorem}
    For $\varphi\in S_N,g\in S_0(\mathbb R^d)\setminus\{0\}$ we have that
    \begin{align*}
        \varphi = \frac{1}{\|g\|^2_{L^2}}\sum\limits_{n\in I_N}\left(\int\limits_{\mathbb T_N}\mathbf{V}_g\varphi(x,\xi)e^{2\pi i \xi^Tn}\mathbf Z_Ng(n-x,-\xi)dxd\xi\right)\epsilon_n.
    \end{align*}
\end{theorem}
\section{The Bargmann Transform on \texorpdfstring{$S_N$}{SN}}
We now restrict ourselves to the study of a specific class of window functions. Choose a matrix $\Omega\in \mathfrak{H}$, a $d\times d$ matrix, and define the dilated Gaussian
$$
h^{\Omega/N}_0(t):=\overline{e^{\pi i t^T(\Omega N^{-1}) t}}.
$$
The following computation from \cite{FlatTori} gives an explicit formula for the STFT (since we have chosen a Gaussian window we now also will refer to this as the Gabor transform) with this window.
\begin{align*}
    \mathbf{V}_{h_0^{\Omega/N}}\epsilon_n(x,\xi ) &=e^{-2\pi i \xi^T n}\mathbf{Z}_N\overline{h_0^{\Omega/N}}{(n-x,\xi)}\\
    & = e^{-2\pi i \xi^T n}\sum\limits_{k\in\mathbb Z^d}e^{\pi i (x+Nk-n)^T\Omega N^{-1} {(x+Nk-n)}}e^{2\pi i \xi^T Nk}\\
    & = e^{-2\pi i \xi^T n}e^{\pi i (x-n)^T \Omega N^{-1} {(x-n)}}\sum\limits_{k\in \mathbb Z^d}\exp(\pi i N k^T\Omega k + 2\pi iNk^T (\xi + \Omega N^{-1}(x-n)))\\
    &=e^{\pi i x^T\Omega N^{-1} x}e^{\pi i n^T\Omega N^{-1} n}e^{2\pi \overline{z_{\Omega /N}}^T n}\vartheta_N(i(\overline{z_{\Omega/ N}}+i\Omega n/N),\Omega).
\end{align*}
$\vartheta_N$ is the Riemann theta function in several variables of order $N$ (c.f. \cite{TataTheta1,Bea}), which is defined by
$$
\vartheta_N(z,{\Omega}):=\sum\limits_{k\in \mathbb Z^d}\exp(\pi iN k^T{\Omega} k+ 2\pi i Nk^Tz),
$$
for $\Omega\in \mathfrak{H}$ and $\overline{z_{\Omega/N}}:=-i(\Omega/N)x-i\xi$. This leads to the following definition of the Bargmann transform which one may think of as a "complexified" Gabor transform. We write
\begin{align}\label{eq:STFTBarg}
\mathbf{B}_{\Omega,N}\varphi(z):=\mathbf {V}_{h_0^{\Omega/N}}\varphi(N(\operatorname{Im}\Omega)^{-1}x,-\mathrm{Re\,}\Omega\,(\operatorname{Im}\Omega)^{-1} x-\xi)&\\ \nonumber 
\times \exp({-\pi Nix^T(\operatorname{Im}\Omega)^{-1}\Omega(\operatorname{Im}\Omega)^{-1}x})&, \quad z=x+i\xi.
\end{align}
This Bargmann type transform may be more directly defined in a way similar to the one in e.g. \cite{LW} since for $\varphi\in S_N$,
$$
\mathbf{B}_{\Omega,N}\varphi({z})=\int\limits_{\mathbb R^d}{\varphi}(t)e^{\pi i t^T(N^{-1}\Omega) t}e^{2\pi z^Tt}.
$$
We remark that it may seem that a $e^{-z^T\frac{\pi}{2} i N\Omega^{-1}z }$ term, is missing when comparing with literature about the Bargmann transform in more "classical" settings (c.f. for example, \cite{FTFA}). Indeed, this convention is very convenient if $\Omega$ has no real part, since then we will directly obtain a canonical multiplier. However, in general, this will result in a rather awkward additional term.

This transform takes its image in a space of theta functions since we have the quasi-periodicity conditions 
\begin{align}
    \mathbf{B}_{\Omega,N}\varphi(z+im) = \mathbf{B}_{\Omega,N}\varphi(z), \quad m\in \mathbb Z^d,
\end{align}
and 
\begin{align}
    \mathbf{B}_{\Omega,N}\varphi(z-i\Omega k) = e^{-\pi i N k^T \Omega k +2\pi N z^Tk}\mathbf{B}_{\Omega,N}\varphi(z), \quad k\in \mathbb Z^d,
\end{align}
since $\vartheta_N$ itself fulfills the quasi-periodicity conditions
\begin{align}
    \vartheta_N(z+m+{\Omega} k,{\Omega}) = e^{-\pi N i k^T{\Omega} k-2\pi Ni z^Tk}\vartheta_N(z,{\Omega}), \quad m,k\in \mathbb Z^d, {\Omega}\in \mathfrak{H}.
\end{align}
Thus we get the system of multipliers
\begin{align}
    e_{-i\Omega k+im}^N(z)=e^{-\pi Ni k^T \Omega k +2\pi Nz^Tk},
\end{align}
which is raising the multipliers 
\begin{align}
    e_{-i\Omega k+im}(z)=e^{-\pi i k^T \Omega k +2\pi z^Tk},
\end{align}
to the $N$-th power. This is associated to the Hermitian form $H(z,w) = N\bar z^T(\operatorname{Im}\Omega)^{-1} w$ and the semicharacter $\chi(im-i\Omega k):=e^{i\pi Nm^Tk}$. Note however, that the associated canonical multiplier is not quite the one we have obtained. The difference is a factor of
\[
\exp\left(-N\frac{\pi }{2}z^T(\operatorname{Im}\Omega)^{-1}z\right)
\]
However, this is firstly a non-vanishing holomorphic function, meaning that our theta functions are equivalent. This will also affect the structure of our $L^2$ space slightly, now for our $L^2$ norm the weight function $\phi$ is replaced by
\[
\phi(z)+\pi N\mathrm{\,Re\,}\left(z^T(\operatorname{Im}\Omega)^{-1}z\right) = 2\pi Nx^T(\operatorname{Im}\Omega)^{-1} x,
\]
in order to accommodate this additional factor. However, this is precisely what we want, since a short computation gives
\begin{align}\label{eq:L2Equivalence}
\langle \mathbf {V}_{h_0^{\Omega/N}}\varphi_1,\mathbf {V}_{h_0^{\Omega/N}}\varphi_2\rangle_{L^2} = \int_{\mathbb T_\Omega}\mathbf{B}_{\Omega,N}\varphi_1(z)\overline{\mathbf{B}_{\Omega,N}\varphi_2(z)}e^{-2\pi N x^T(\operatorname{Im}\Omega)^{-1} x}dV_{\mathbb T,N}.
\end{align}
Now we have the following result which gives a nice basis for this space of theta functions.
\begin{proposition}
    The family of functions
    \[
    \left\{\frac{\mathbf{B}_{\Omega,N}\epsilon_n}{\|\mathbf{B}_{\Omega,N}\epsilon_n\|_{L^2}}\right\}_{n\in I_N}
    \]
    gives an orthonormal basis for $\Theta(e_{-i\Omega k+im}^N)$. In particular, the Bargmann transform is a vector space isomorphism onto $\Theta(e_{-i\Omega k+im}^N)$.
\end{proposition}
\begin{proof}
    The Moyal identity gives immediately that these functions must be orthogonal. Moreover, Theorem \ref{thm:ThetaDim} tells us that these must form a spanning set.
\end{proof}
\section{The Proof of the Frame Results}
We now turn to the frame results for the window $h^{\Omega/N}_0$. Recall that our target is to find sets for which we get a frame for the DGT with window
\[
\mathbf h_0^{\Omega/N}:=\mathbf P^N_{samp}(h^{\Omega/N}_0).
\]
Recall that $\{(k_j,l_j)\}_{j=1}^{K}\subseteq I_N^2$ gives a frame when we can find constants $A,B>0$ with 
\[
A\|\mathbf f\|^2\le \sum\limits_{j=1}^K|\mathbf{V}_{\mathbf h_0^{\Omega/N}}\mathbf f[k_j,l_j]|^2 \le B\|\mathbf f\|^2
\]
for any $\mathbf f\in \mathbb C^{N^d}$. First we relate this to frames on $S_N$.
\begin{lemma}\label{le:DGTSTFTFrame}
    The set $\{(k_j,l_j)\}_{j=1}^{K}\subseteq I_N^2$ induces a frame for the DGT with window $\mathbf h_0^{\Omega/N}$, if and only if $\{(k_j,\frac{l_j}{N})\}_{j=1}^{K}$ gives a frame for the STFT on $S_N$ with window $h^{\Omega/N}_0$, i.e. there exists constants $A,B>0$ such that for any $\varphi\in S_N$ we have
    \[
    A\|\varphi\|^2_{S_N}\le \sum\limits_{j=1}^K|\mathbf{V}_{ h_0^{\Omega/N}}\varphi(k_j,\frac{l_j}{N})|^2 \le B\|\varphi\|^2_{S_N}
    \]
\end{lemma}
\begin{proof}
    Since we have by \cite{FlatTori} that the operators $\mathbf{\Sigma}_N$ and $\mathbf{P}^N_{samp}$ are surjective, we get by Theorem \ref{thm:DGTtoSN} that the frame conditions are equivalent. In particular, we have that the frame bounds $A,B$ are the same in both cases if we take the euclidean norm on the DGT side.
\end{proof}
Now we want to move from the problem in $S_N$ to a problem about theta functions. We have the following.
\begin{lemma}\label{le:STFTBargFrame}
    The following are equivalent.
    \begin{itemize}
        \item[\textbf{(a)}] A finite set of points $\{(x^{(j)},\xi^{(j)})\}_{j\in J}$ induces a frame for the STFT on $S_N$ with window $h^{\Omega/N}_0$.
        \item[\textbf{(b)}]  The set of points $\{z^{(j)}\}_{j\in J}$ where $z^{(j)} : = -iN^{-1}\Omega x^{(j)}-i\xi^{(j)}$ is a frame for the Bargmann transform on $S_N$, i.e. there exist $A,B>0$ such that for every $\varphi\in S_N$ we have 
        \[
            A\|\varphi\|^2_{S_N}\le \sum\limits_{j=1}^K|\mathbf{B}_{\Omega,N}\varphi(z^{(j)})|^2 \le B\|\varphi\|^2_{S_N}
        \]
        \item[\textbf{(c)}] The set of points $\{z^{(j)}\}_{j\in J}$ is not contained in the vanishing set of any single nonzero $\theta\in \Theta(e_{-i\Omega k+im}^N)$.
    \end{itemize}
\end{lemma}
\begin{proof}
    \textbf{(a) is equivalent to (b):} This follows rather directly from (\ref{eq:STFTBarg}), since the term we multiply the STFT by to obtain the Bargmann transform has bounds strictly larger than 0 and less than infinity in a single fundamental domain.
    \\~~\\
    \textbf{(b) implies (c):} We argue by contraposition. If the zero set of a single function $\theta\in \Theta(e_{-i\Omega k+im}^N)$ contains all the $z^{(j)}$, then $\sum\limits_{j=1}^K|\theta(z^{(j)})|^2 =0$. Thus the $\varphi\in S_N$ with $\mathbf{B}_{\Omega,N}\varphi = \theta$ will violate the lower frame bound no matter how we choose $A$.
    \\~~\\
    \textbf{(c) implies (b):} Note that the map 
    \[
    \varphi\mapsto \sqrt{ \sum\limits_{j=1}^K|\mathbf{B}_{\Omega,N}\varphi(z^{(j)})|^2}
    \]
    on $S_N$ defines a seminorm. But if \textbf{(c)} holds, it will also be strictly positive for any nonzero $\varphi\in S_N$, thus making it a genuine norm. But since it is a norm on the finite dimensional vector space $S_N$, it will be equivalent to $\|\cdot \|_{S_N}$. This norm equivalence is just the square root of the frame bound.
\end{proof}
Now there is little left to do in order to prove our first main theorem.
\begin{proof}[Proof of Theorem A]
    First note that the translational stabilizer of the multipliers $(e_{-i\Omega k-im})$ is just the lattice point. To see this, first observe that 
    \[
    c_1(z,-i\Omega k-im) = \mathrm{Im}(\bar z^T(\operatorname{Im}\Omega)^{-1}
    {(-i\Omega k-im)}),\quad k,m\in \mathbb Z^d,z\in \mathbb C^d.
    \]
    Thus for $z=x+i\xi$, and first setting $k=0$ this implies that for $z$ to represent an element of the translational stabilizer,
    \[
    x^T(\operatorname{Im}\Omega)^{-1}m
    \]
    must take integral values, or equivalently, $x \in \mathrm{Im\,\Omega}\cdot \mathbb Z^d$. Similarly, now setting $m=0$ we see that we also need 
    \[
    x^T(\operatorname{Im}\Omega)^{-1}\mathrm{\,Re\,}\Omega k+\xi^Tk\in \mathbb Z.
    \]
    Since we already know that $a:=(\mathrm{Im\,\Omega)^{-1}}x$ is integer valued for a point in the translational stabilizer, we see that this gives
    \[
    \xi = m-\mathrm{Re\,\Omega}a,\quad a,m\in \mathbb Z^d
    \]
    and consequently,
     $x$ and $\xi$ together describe precisely the lattice $\Lambda_\Omega$ and we have that $\Lambda_\Omega^\circ/\Lambda_\Omega =\{0\}+\Lambda_\Omega$.
    
    Now by Lemmas \ref{le:DGTSTFTFrame} and \ref{le:STFTBargFrame} finding a frame for the DGT is equivalent to finding subsets of the translational stabilizer for $\Theta(e_{-i\Omega k+im}^N)$ (note that these are precisely the points in $\mathbb T_\Omega$ corresponding to $I_N^2$ for the DGT) that are not contained in the vanishing set of any single element of $\Theta(e_{-i\Omega k+im}^N)$. But Lemma \ref{le:ThetaZeroes} gives the structure of the vanishing sets of certain elements in our space of theta functions. Thus we see that we get no frame if all of our points lie in a set of the form 
    \[
    \{\theta=0\} = \bigcup_{i=1}^Nt_{v_i}(A),
    \]
    where $A$ is the vanishing set of a nontrivial element in $\Theta(e_{-i\Omega k+im})$. But now Theorem \ref{thm:ThetaDim} tells us that $\Theta(e_{-i\Omega k+im})$ has precisely one basis element. But this element is precisely $\vartheta_1(iz,\Omega)$. (To see this, compare the multipliers of $\vartheta_1$ and of $\mathbf{B}_{\Omega,1}\epsilon_0$).
\end{proof}
\begin{proof}[Proof of Theorem 1.2]
    In dimension one we can say much more. Firstly we know precisely what the vanishing set of $\vartheta_1$ is. By for example, \cite[Chapter 6.3.1]{Dondaldson}, we have that the vanishing set is a simple zero in the center of the fundamental domain. We denote this point by $z_0$.

    \cite[Lemma 15]{Dondaldson} also gives us (as a simple consequence of the Rouch\'e Theorem) that the vanishing set of any nonzero element in $\Theta(e_{-i\Omega k+im}^N)$ is composed of $N$ zeroes counted with multiplicity on the torus. Thus we immediately get a frame for any collection of more than $N$ points. 

    Now if we observe a set $\{z^{(j)}\}_{j=1}^N$ of exactly $N$ points which does not induce a frame (i.e. the set is equal to the vanishing set of some $\theta\in \Theta(e_{-i\Omega k+im}^N)$). Proposition \ref{prop:EquiConverse} now gives us that $\{e_{-i\Omega k+im}^N\}$ must be equivalent to $\{\prod_{j=1}^Ne_{-i\Omega k+im}(\cdot -z^{(j)}+z_0)\}$. But this equivalence means by the theorem of the square, that 
    \begin{align}\label{eq:FrameShift}
    \sum_{j=1}^N(z^{(j)}-z_0) +\Lambda_\Omega
    \end{align}
    lies in the translational stabilizer. Now all that is left to do is to write this condition explicitly. As mentioned before, we have that $z_0=\tfrac{-i}{2}(1+\Omega)+\Lambda_\Omega.$ Thus we have that "(\ref{eq:FrameShift}) lies in the translational stabilizer" becomes 
    \[
    \frac{1}{N}\sum_{j=1}^Nk_j-\frac{N}{2}\in \mathbb Z,\qquad \frac{1}{N}\sum_{j=1}^Nl_j-\frac{N}{2}\in \mathbb Z
    \]
    in the notation of Theorem \ref{thm:1dCase}. This is never fulfilled for odd $N$ and for even $N$ precisely when 
    \[
    \sum_{j=1}^Nk_j\in N\mathbb Z\qquad \sum_{j=1}^Nl_j\in N\mathbb Z.
    \]
    Together with the converse statement which comes from Theorem A, we have proved the $K=N$ case.
 
    Lastly, we know that we cannot get a frame for less than $N$ points, since we know the image space of the STFT to have dimension $N$, thus a frame must be induced by at least $N$ distinct points.
\end{proof}
\section{Asymptotics for the Localization Operators}
We begin by examining Localization operators on $S_N$, rather than discrete localization operators. We define these for a measurable and bounded function $a$ on the $[0,1)^{2d}$ as
\begin{align}
    \mathbf{R}^{(N)}_{a,g}\varphi:=\frac{1}{\|g\|^2_2}\sum\limits_{n\in I_N}\left(\int\limits_{\mathbb T_N}a_N(x,\xi)\cdot\mathbf{V}_g\varphi(x,\xi)e^{2\pi i \xi^Tn}\mathbf Z_Ng(n-x,-\xi)dxd\xi\right)\epsilon_n,\quad \varphi\in S_N,
\end{align}
where $a_N(x,\xi):=a(x/N,\xi)$
\subsection{Relation to Toeplitz Operators}
In order to study localization operators, we will also relate them to the corresponding complexified operators. These are the Toeplitz operators. They are defined as follows.
\begin{definition}
    For a measurable and bounded function $\tilde a$ on the torus $\mathbb T_\Omega$ we define the Toeplitz operators as
    \[
    \mathbf T_{\tilde a,N}:\Theta(e_\lambda^N)\to \Theta(e_\lambda^N),\quad \theta \mapsto P^N(\tilde a\cdot \theta)
    \]
    where $P^N$ denotes the orthogonal projection from $L^2$ functions with quasi-periodicity given by $e_\lambda^N$ to $\Theta(e_\lambda^N)$ (with respect to the $L^2$ norm we had defined for $\Theta(e_\lambda^N)$.)
\end{definition}
Since for any $\theta_1,\theta_2\in \Theta({e_\lambda^N})$ we have that
\[
\langle \mathbf T_{\tilde a,N}\theta_1,\theta_2\rangle_{L^2} = \langle  \tilde a\cdot \theta_1,\theta_2\rangle_{L^2},
\]
we know that the RHS of this equation gives the dual formulation of the Toeplitz operator. It is this dual formulation which will give us the connection to time-frequency localization operators. Recall that these were defined by
\[
\mathbf{R}^{(N)}_{a,g}\varphi:=\frac{1}{\|g\|^2_{L^2(\mathbb R^d)}}\sum\limits_{n\in I_N}\left(\int\limits_{\mathbb T_N}a_N\cdot\mathbf{V}_g\varphi(x,\xi)e^{2\pi i \xi^Tn}\mathbf Z_Ng(n-x,-\xi)dxd\xi\right)\epsilon_n,\quad \varphi\in S_N.
\]
Thus we have by (\ref{eq:STFTeps}) for $\varphi_1,\varphi_2\in S_N$ that
\begin{align*}
    \langle\mathbf{R}^{(N)}_{a,g}\varphi_1,\varphi_2\rangle_{S_N} = \frac{1}{\|g\|^2_{L^2(\mathbb R^d)}}\langle a_N\cdot \mathbf V_g\varphi_1,\mathbf V_g\varphi_2\rangle_{L^2}.
\end{align*}
Now if we define $F_\Omega(x,\xi):=-i({\Omega}x+\xi)$ and $\tilde a:=a\circ F_\Omega^{-1}$ we have by (\ref{eq:L2Equivalence}) that for $g=h^{\Omega/N}_0$,
\[
\langle\mathbf{R}^{(N)}_{a,g}\varphi_1,\varphi_2\rangle_{S_N} = \frac{1}{\|h^{\Omega/N}_0\|^2_{L^2(\mathbb R^d)}}\langle \tilde a\mathbf B_{\Omega,N}\varphi_1,\mathbf{B}_{\Omega,N}\varphi_2\rangle_{L^2}.
\]
It follows that the spectrum of $\mathbf{R}^{(N)}_{a,g}$ is equal to that of $\mathbf T_{\tilde a,N}$. (The factor $1/\|h^{\Omega/N}_0\|^2_{L^2(\mathbb R^d)}$ in the formula above will be removed when renormalizing the Bargmann transform).
\subsection{Asymptotics of the Bergman Function}
Although for most of the results about Toeplitz operators we will simply refer to \cite{BerToep}, we want to give a proof of the fundamental observation which enables the the results in what follows here, namely the asymptotic of the Bergman function. In particular, the proof in the case we look at here can be simplified, and so we hope that the proof we give here will serve as some convenience to the reader. The following is essentially \cite[Theorem 2.1]{BerToep} applied to our context.
\begin{theorem}
    Let $B_N$ denote the Bergman function of $\Theta(e_{-i\Omega k+im}^N)$. Then we have that
    \[
    B_N \to 1
    \]
    in $L^1(dV_{\mathbb T,1})$ (note however, that the volume forms used to define $B_N$ do vary with $N$).
\end{theorem}
\begin{proof}
    \textbf{Step 1: Upper Bound.} We use the extremal characterization of the Bergman kernel to write
    \[
   B_N(z) = \sup_{\theta\in \Theta(H,\chi)\setminus\{0\}}\frac{|\theta(z)|^2e^{-\phi(z)}}{\|\theta\|_{L^2}^2}.
    \]
    (Here we will take $\phi(z)= \pi N  z^T(\operatorname{Im}\Omega)^{-1} \overline z$.) Note that using the canonical representation of theta functions will not actually affect the Bergman function since any change in the functions is absorbed by the change in the weight function. But since $\Theta(H,\chi)$ has finite dimension, we have that this supremum is attained by an extremal element $\tilde\theta$. Then we have for any $R_N$ such that the domain $E_{R_N}(z):=z+(\operatorname{Im}\Omega)^{1/2}B(0,R_N)$ fits in to a fundamental domain of the torus, where $B(0,R_N)$ is the ball of radius $R_N$ around $0$.
    \[
    B_N(z) = \frac{|\tilde \theta(z)|^2e^{-\phi(z)}}{\int_{\mathbb T_\Omega}|\tilde \theta|^2e^{-\phi}dV_{\mathbb T,N}}\le\frac{|\tilde \theta(z)|^2e^{-\phi(z)}}{\int_{E_{R_N}(z)}|\tilde\theta|^2e^{-\phi}dV_{\mathbb T,N}}.
    \]
    Recentering the term on the RHS at $0$ via the formula
    \[
    \phi(z+w) = \phi(z)+\phi(w)+2\pi N\mathrm{Re}(w^T(\operatorname{Im}\Omega)^{-1}\bar z) 
    \]
    we get the same factor $\phi(z)$ in the numerator and denominator, and can absorb $2\pi N\mathrm{Re}(w^T(\operatorname{Im}\Omega)^{-1}\bar z)$ into the theta function. Thus it is no loss of generality to assume that $z=0.$
    Since $\tilde\theta$ is holomorphic, we have that $|\tilde \theta|^2$ is subharmonic and we have by the sub-mean value inequality that 
    \[
    \frac{|\tilde \theta(0)|^2e^{-\phi(z)}}{\int_{E_{R_N}(0)}|\tilde\theta|^2e^{-\phi}dV_{\mathbb T,N}}\le  \frac{e^{-\phi(0)}}{\int_{E_{R_N}(0)}e^{-\phi}dV_{\mathbb T,N}}
    \]
    
    Now set $R_N$ to be a sequence with $R_N\to 0 $ and $R_N\sqrt N \to \infty$ for $N\to \infty$. Then we apply a change of variables: multiply $w$ by $N^{-1/2}(\operatorname{Im}\Omega)^{1/2}$. Then we get that
    \[
    \int_{E_{R_N}(0)}e^{-\phi(w)}dV_{\mathbb T,N}(w) = \int_{B(0,{R_N\sqrt N})}e^{-\pi w^T\overline{w}}dV(w)
    \]
    where $dV$ is the standard Lebesgue volume element. Now by dominated convergence, this integral tends to 1, thus we have that at every point,
    \[
    \limsup_{N\to \infty}B_N(z)\le 1. 
    \]
    \textbf{Step 2: Lower Bound}
    We now seek a lower bound for the integral of the Bergman function. But this is just 
    \[
    \int_{\mathbb T_\Omega}B_NdV_{\mathbb T,N} = \dim \Theta(H,\chi) =N^d.
    \]
    We remark that since we only need a lower bound here, the Moyal formula would actually be sufficient. Thus the proof is complete once we combine this with \cite[Lemma 2.2]{BerToep} which we state below for the reader's convenience.
\end{proof}
\begin{lemma}[Lemma 2.2 in \cite{BerToep}]
    Assume that $(X,\mu)$ is a finite measure space and that $f_N$ is a uniformly bounded sequence of bounded functions and $f$ is a bounded function. If
    \[
    \lim_{N\to \infty}\int_Xf_Nd\mu =\int_Xf d\mu
    \]
    and 
    \[
    \limsup_{N\to \infty}f_N\le f
    \]
    then $f_N\to f$ in $L^1$.
\end{lemma}
We remark that the asymptotic for $B_N$ is essentially just a consequence of the asymptotic expansion of the Bergman kernel from \cite{Catlin, Zelditch98} Indeed, in essence we have simply computed the leading term for this asymptotic expansion.

We can also like in \cite{BerToep} pull a corollary about Toeplitz operators directly. This can directly be applied to localization operators on $S_N$ once we change variables back to the Lebesgue measure.
\begin{corollary}[Corollary 2.3 in \cite{BerToep}]
    We have that for any bounded symbol $a$ on $\mathbb T_\Omega$,
    \[
    N^{-d}\mathrm{Tr}\,\mathbf{T}_{a,N}\to \int_{\mathbb T_{\Omega}}adV_{\mathbb T,1} 
    \]
\end{corollary}
\begin{proof}
    Since the trace of the Toeplitz operator is expressed by
    \[
    \sum_{i}\langle a \theta_i,\theta_i\rangle_{L^2} = \int_{\mathbb T_\Omega}aB_NdV_{\mathbb T,N}, 
    \]
    we get the theorem directly from $L^1$ convergence.
\end{proof}
Now since we cannot even in this case, genuinely simplify the exposition leading to the second statement in Theorem B we simply refer to \cite{BerToep} for the proof. 
\begin{theorem}[Theorem 2.6 in \cite{BerToep}]
    Let $a$ be a real valued, bounded function, and let $N_{<\alpha}$ be the function that counts how many eigenvalues of a given operator are less than $\alpha$. Then we have that for any real valued symbol $a$ and all but countably many $\alpha$,
    \[
    \lim_{N\to \infty}N^{-d}N_{<\alpha}(\mathbf T_{a,N}) = \int_{\{a<\alpha\}}dV_{\mathbb T,1}
    \]
    with a similar bound for $N_{>\alpha}$.
\end{theorem}
\subsection{Discrete Localization Operators}
Now to obtain Theorem B, we must prove that the trace norms of the localization operators on the Torus approach those of the discrete localization operator. The statement that will imply the theorem is the following.
\begin{proposition}
    We have that for any Riemann integrable symbol $a$ and $g=h_0^{\Omega/N}$, we have for $N\to \infty$
    \begin{align}
        N^{-d}\|\mathbf{R}_{a}^N-\mathbf R^{(N)}_{a,g}\|_{\mathrm{Tr}}\to 0
    \end{align}
    where we consider $\mathbf R^{(N)}_{a,g}$ as an operator on $\mathbb C^{N^d}$ and $\|\cdot\|_{\mathrm{Tr}}$ denotes the trace norm.
\end{proposition}
\begin{proof}
    Recall that we have 
    \begin{align*}
    \langle\mathbf{R}^{(N)}_{a,g}\varphi_1,\varphi_2\rangle_{S_N} = \frac{1}{\|h^{\Omega/N}_0\|^2_{L^2(\mathbb R^d)}}\langle a_N\cdot \mathbf V_g\varphi_1,\mathbf V_g\varphi_2\rangle_{L^2(\mathbb T_N)} = N^d\int_{[0,1)^{2d}}a(x,\xi)Q(x,\xi)(\varphi_1,\varphi_2)dxd\xi
\end{align*}
where 
\[
Q(x,\xi)(\varphi_1,\varphi_2):=\frac{1}{\|h^{\Omega/N}_0\|^2_{L^2(\mathbb R^d)}}\langle\varphi_1,\pi(Nx,\xi)h^{\Omega/N}_0\rangle_{S_0'\times S_0}\cdot\overline{\langle\varphi_2,\pi(Nx,\xi)h^{\Omega/N}_0\rangle}_{S_0'\times S_0}.
\]
Similarly, if we write the discrete operator as an operator on $S_N$ we have
\begin{align*}
 \langle \mathbf{R}^{N} \varphi_1,\varphi_2\rangle &= N^{-d}\sum_{k,l\in I_N}a(k/N,l/N)P(k,l)(\varphi_1,\varphi_2) 
\end{align*}
where
\[
P(k,l)(\varphi_1,\varphi_2)=\frac{1}{|\mathbf{h}^{\Omega/N}_0|^2}\langle \varphi_1,\pi (k,l/N){h}^{\Omega/N}_0\rangle_{S_0'\times S_0}\cdot\overline{\langle \varphi_2,\pi(k,l/N){h}^{\Omega/N}_0\rangle}_{S_0'\times S_0},
\]
when we use the identification $S_N\cong \mathbb C^{N^d}$. In order to really compare the discrete operator with the one on $S_N$ however, we will need to "quantize" the continuous data. To this end we introduce $a^{samp}$, which presents a sort of quadrature approximation of $a$. More precisely, if we partition $[0,1)^{2d}$ into the cells
\[
S_{k,l}:=\prod_{j=1}^d\left[\frac{k_j-\frac{1}{2}}{N},\frac{k_j+\frac{1}{2}}{N}\right)\times \prod_{j=1}^d\left[\frac{l_j-\frac{1}{2}}{N},\frac{l_j+\frac{1}{2}}{N}\right)\quad k,l\in I_N
\]
(interpreted periodically) then $a^{samp}$ is defined as
\[
a^{samp}:=\sum_{k,l}a\left(\frac{k}{N},\frac{l}{N}\right)\mathbf{1}_{S_{k,l}}
\]
We similarly make $P$ into a function on $[0,1)$ by extending each value to its cell. Now we have as a consequence that
\[
\mathbf{R}_{a}^N-\mathbf R^{(N)}_{a,g} = N^d\int_{[0,1)^{2d}}(a^{samp}-a)P+(aP-aQ)dxd\xi.
\]
Now since $\|P\|_{\mathrm{Tr}}=1$, we have that
\[
N^{-d}\|\mathbf{R}_{a}^N-\mathbf R^{(N)}_{a,g}\|_{\mathrm{Tr}}\le \|a-a^{samp}\|_{L^1}+\|a\|_{L^1}\sup_{x,\xi\in[0,1)^{2d}}\|P(x,\xi)-Q(x,\xi)\|_{\mathrm{Tr}}.
\]
Clearly, what is left to do is estimate $\|P-Q\|_{\mathrm{Tr}}$. Here we may once again apply the asymptotic behavior of the Bergman kernel. To this end, define for $z\in \mathbb C^d$ an $\varepsilon_z\in S_N$ by the identity
\[
\mathbf{B}_{\Omega,N}\varepsilon_{z}=\|h_0^{\Omega/N}\|_{L^2}\cdot \frac{K(\cdot ,z)}{\sqrt{K(z,z)}}.
\]
This is always well defined for $N\ge2$ (since the associated line bundle will be base point free the kernel will never vanish on the diagonal, see \cite[Theorem 3.7]{Bea}). Then we have by the reproducing kernel property that
\[
C_\zeta(z):=|\langle \varepsilon_{z},\varepsilon_\zeta\rangle_{S_N}|^2 = \|h_0^{\Omega/N}\|^{-4}_{L^2}|\langle \mathbf{B}_{\Omega,N}\varepsilon_{z},\mathbf{B}_{\Omega,N}\varepsilon_\zeta\rangle_{L^2}|^2 = \frac{|K(z,\zeta)|^2}{K(z,z)K(\zeta,\zeta)}\quad z,\zeta \in \mathbb C^d.
\]

Now we have by \cite[Theorem 1.1]{BerUniversality} that for $N\to \infty$
\begin{align}\label{eq:BergAsOffDiag}
K(N^{-1/2}z,N^{-1/2}\zeta)\to e^{\frac{\pi}{2}  (z+\bar \zeta)^T\mathrm{(\Im\,\Omega})^{-1}(z+\bar \zeta)},
\end{align}
locally uniformly. Thus it follows that
\[
C_{\zeta N^{^{-1/2}}}(zN^{^{-1/2}}) \to e^{-\pi\overline{(z-\zeta)}^T(\mathrm{Im\,\Omega})^{-1}(z-\zeta)},
\]
and consequently for sufficiently large $N$ and $z,\zeta\in \mathbb C^d$ where $\zeta$ is fixed and $z$ is such that $|z-\zeta|<R$ for some $R>0$, we have that there is a sequence $\eta_N\downarrow0$ with
\[
0\le 1-C_{\zeta N^{^{-1/2}}}(zN^{^{-1/2}})\le 1 -e^{-\pi\overline{(z-\zeta)}^T(\mathrm{Im\,\Omega})^{-1}(z-\zeta)}+\eta_N\le C|z-\zeta|^2+\eta_N.
\]
This will imply the desired estimate once we translate it back to the real setting. First, it will be enough to consider the operator $P$ evaluated at 0. In this case, (writing $P$ as a linear operator rather than a bilinear one) we have
\[
P(0,0)\varphi=\langle\varphi, \varepsilon_0\rangle_{S_N}\varepsilon_0.
\]
Moreover, we consider $(x_N,\xi_N)\in [0,1)^{2d}$ to be inside the cell of $(0,0)$, i.e. for all $i=1,....,d$ we have $|x_{N,i}|<\tfrac{1}{2N}$ and $|\xi_{N,i}|<\tfrac{1}{2N}$. We will also denote $z_N:=F_\Omega(x_N,\xi_N)$. Then  we have firstly that (compare with the formula for the reproducing kernel for the Gabor space in \cite{FlatTori})
\[
Q(x_N,\xi_N) = B_N(z)\langle\cdot, \varepsilon_z\rangle_{S_N}\varepsilon_z,
\]
and also that for large $N$,
\[
1-C_0(z_N)\le \frac{C}{N}+\eta_N
\]
where the constants are independent of $z_N$ so long as it runs in the cell of $(0,0)$. Thus we may estimate
\begin{align*}
    \|P(0,0)-Q(x,\xi)\|_{\mathrm{Tr}}&=\|\langle\cdot, \varepsilon_0\rangle_{S_N}\varepsilon_0-B_N(z)\langle\cdot, \varepsilon_z\rangle_{S_N}\varepsilon_z\|_{\mathrm{Tr}}\\
    &\le \|\langle\cdot, \varepsilon_0\rangle_{S_N}\varepsilon_0-\langle\cdot, \varepsilon_z\rangle_{S_N}\varepsilon_z\|_{\mathrm{Tr}}+|B(z)-1|\cdot \|\langle\cdot, \varepsilon_z\rangle_{S_N}\varepsilon_z\|_{\mathrm{Tr}}\\
    &=\|\langle\cdot, \varepsilon_0\rangle_{S_N}\varepsilon_0-\langle\cdot, \varepsilon_z\rangle_{S_N}\varepsilon_z\|_{\mathrm{Tr}}+|B(z)-1|.
\end{align*}
Using the formula (see e.g. \cite[Eq. 1.186]{Watrous2018})
\begin{align*}
    \|vv^*-uu^*\|_{\mathrm{Tr}}=2\sqrt{1-|\langle u,v\rangle|^2},
\end{align*}
we get furthermore,
\begin{align*}
    \|\langle\cdot, \varepsilon_0\rangle_{S_N}\varepsilon_0-\langle\cdot, \varepsilon_z\rangle_{S_N}\varepsilon_z\|_{\mathrm{Tr}}+|B(z)-1|&\le 2\sqrt{1-C_0(z)}+|B(z)-1|.
\end{align*}
Plugging in a given $(x_N,\xi_N)$, we get
\begin{align*}
     \|P(0,0)-Q(x_N,\xi_N)\|_{\mathrm{Tr}}\le 2\sqrt{\frac{C}{N}+\eta_N}+|B(z_N)-1|.
\end{align*}
Applying (\ref{eq:BergAsOffDiag}) to the diagonal, we see that we may also find $\delta_N\downarrow0$ with 
\[
|B(z_N)-1|\le \delta_N.
\]
This can again be chosen independently of $z_N$, since we again have locally uniform convergence. Setting 
\[
c_N:= 2\sqrt{\frac{C}{N}+\eta_N}+\delta_N
\]
we have 
\[
\sup_{(x_N,\xi_N)\in S_{0,0}}\|P(0,0)-Q(x_N,\xi_N)\|_{\mathrm{Tr}}\le c_N.
\]
In order to see that this also holds for general $(k,l)$ use that
\[
P(k,l) = \pi (k,l/N)P(0,0)\pi(k,l/N)^*.
\]
We have a similar shift identity for $Q$. Since the time frequency shift is unitary, we have that it will not affect the trace norm. Thus we have in total,
\[
N^{-d}\|\mathbf{R}_{a}^N-\mathbf R^{(N)}_{a,g}\|_{\mathrm{Tr}}\le \|a-a^{samp}\|_{L^1}+c_N\|a\|_{L^1}.
\]
but since $a$ is Riemann integrable, we have that $\|a-a^{samp}\|_{L^1}$ also converges to zero. Thus we are done.
\end{proof}
\begin{proof}[Proof of Theorem B]
    Applying Mirsky's eigenvalue inequality \cite{Mirsky} we have,
    \[
    D_N:=N^{-d}\sum_{j=1}^{N^d}|\lambda_j(\mathbf{R}_{a}^N)-\lambda_j(\mathbf R^{(N)}_{a,g})|\le N^{-d}\|\mathbf{R}_{a}^N-\mathbf R^{(N)}_{a,g}\|_{\mathrm{Tr}}\to 0.
    \]
    where $\lambda_j(A)$ denotes the $j$-th eigenvalue of the operator $A$. Now fix $\alpha,\delta>0$ and observe that for large $N$ there can be at most $N^dD_N/\delta$ eigenvalues with 
    \[
    |\lambda_j(\mathbf{R}_{a}^N)-\lambda_j(\mathbf R^{(N)}_{a,g})|\ge \delta.
    \]
    Thus we have
    \[
    N_{<\alpha-\delta}(\mathbf{R}_{a,g}^{(N)})-\frac{N^dD_N}{\delta}\le  N_{<\alpha}(\mathbf{R}_{a}^N)\le  N_{<\alpha+\delta}(\mathbf{R}_{a,g}^{(N)})+\frac{N^dD_N}{\delta}.
    \]
    Observing the large $N$ limit we have for all but countably many $\alpha\pm\delta$ that
    \[
    \mathrm{Vol}(a<\alpha-\delta)\le\liminf_{N\to \infty}N^{-d}N_{<\alpha}(\mathbf{R}_{a}^{N})\le \limsup_{N\to \infty}N^{-d}N_{<\alpha}(\mathbf{R}_{a}^{N})\le \mathrm{Vol}(a<\alpha+\delta).
    \]
    Sending $\delta\to 0$, we obtain
     \[
    \mathrm{Vol}(a<\alpha)\le\liminf_{N\to \infty}N^{-d}N_{<\alpha}(\mathbf{R}_{a}^{N})\le \limsup_{N\to \infty}N^{-d}N_{<\alpha}(\mathbf{R}_{a}^{N})\le \mathrm{Vol}(a\le\alpha).
    \]
    Thus we get that 
    \[
    \lim_{N\to \infty}N^{-d}N_{<\alpha}(\mathbf{R}_{a}^{N})=\mathrm{Vol}(a<\alpha)
    \]
     so long as $\mathrm{Vol}(a=\alpha)=0$, which is the case for all but countably many $\alpha$.
\end{proof}

\subsection{A.I. Declaration}
ChatGPT 6.0 Astra was consulted in the editing of this article and found several inaccuracies and typos in previous versions. Additionally, the proof of Proposition 6.5 arose from a conversation with ChatGPT. The author has checked all mathematical statements originating from LLMs and takes responsibility for the mathematical assertions in the paper.

\bibliography{ref}{}

@article{FlatTori,
title = {Time-frequency analysis on flat tori and {G}abor frames in finite dimensions},
journal = {Applied and Computational Harmonic Analysis},
volume = {69},
pages = {101622},
year = {2024},
issn = {1063-5203},
doi = {https://doi.org/10.1016/j.acha.2023.101622},
url = {https://www.sciencedirect.com/science/article/pii/S1063520323001094},
author = {L.D. Abreu and P. Balazs and N. Holighaus and F. Luef and M. Speckbacher}
}

@incollection{Bea,
  TITLE = {{Theta functions, old and new}},
  AUTHOR = {Beauville, A.},
  URL = {https://hal.science/hal-01141595},
  BOOKTITLE = {{Open Problems and Surveys of Contemporary Mathematics}},
  PUBLISHER = {{Higher Education Press et International Press}},
  SERIES = {Surveys of Modern Mathematics},
  VOLUME = {6},
  PAGES = {99-131},
  YEAR = {2013},
  HAL_ID = {hal-01141595},
  HAL_VERSION = {v1},
}

@article{BerToep,
	author = {Berman, R.},
	date = {2006/03/01},
	doi = {10.1007/BF02930984},
	id = {Berman2006},
	isbn = {1559-002X},
	journal = {The Journal of Geometric Analysis},
	number = {1},
	pages = {1--22},
	title = {Super {T}oeplitz operators on line bundles},
	url = {https://doi.org/10.1007/BF02930984},
	volume = {16},
	year = {2006}}

@ARTICLE{Daub,
  author={Daubechies, I.},
  journal={IEEE Transactions on Information Theory}, 
  title={Time-frequency localization operators: a geometric phase space approach}, 
  year={1988},
  volume={34},
  number={4},
  pages={605-612},
  doi={10.1109/18.9761}}

@book{TataTheta1,
  title={Tata {L}ectures on {T}heta {I}},
  author={Mumford, D. and Musili, C. and Nori, M. and Previato, E. and Stillman, M.},
  isbn={9781489928436},
  series={Progress in Mathematics},
  url={https://books.google.de/books?id=IHP1BwAAQBAJ},
  year={2013},
  publisher={Birkh{\"a}user Boston}
}

@book{FTFA,
  title={Foundations of {T}ime-{F}requency {A}nalysis},
  author={Gr{\"o}chenig, K.},
  isbn={9780817640224},
  lccn={00044508},
  series={Applied and Numerical Harmonic Analysis},
  url={https://books.google.de/books?id=sjN2qq99-WwC},
  year={2001},
  publisher={Birkh{\"a}user Boston}
}

@book{BirkenhakeLange,
  title={Complex {A}belian {V}arieties},
  author={Birkenhake, C. and Lange, H.},
  isbn={9783540204886},
  lccn={2004045272},
  series={Grundlehren der mathematischen Wissenschaften},
  url={https://books.google.se/books?id=MOW2gEP7HIkC},
  year={2004},
  publisher={Springer Berlin Heidelberg}
}

@article{LW,
	author = {Luef, F. and Wang, X.},
	date = {2023/06/01},
	doi = {10.1007/s00039-023-00640-z},
	id = {Luef2023},
	isbn = {1420-8970},
	journal = {Geometric and Functional Analysis},
	number = {3},
	pages = {778--823},
	title = {Gaussian {G}abor frames, {S}eshadri constants and generalized {B}user--{S}arnak invariants},
	url = {https://doi.org/10.1007/s00039-023-00640-z},
	volume = {33},
	year = {2023}}

@book {Dondaldson,
    AUTHOR = {Donaldson, S.},
     TITLE = {Riemann surfaces},
    SERIES = {Oxford Graduate Texts in Mathematics},
    VOLUME = {22},
 PUBLISHER = {Oxford University Press, Oxford},
      YEAR = {2011},
     PAGES = {xiv+286},
      ISBN = {978-0-19-960674-0},
   MRCLASS = {30-01 (14H55 30F10 30F30 32G15 57M12)},
  MRNUMBER = {2856237},
MRREVIEWER = {Alexander\ Zvonkin},
       DOI = {10.1093/acprof:oso/9780198526391.001.0001},
       URL = {https://doi.org/10.1093/acprof:oso/9780198526391.001.0001},
}

@article {Lindholm,
    AUTHOR = {Lindholm, N.},
     TITLE = {Sampling in weighted {$L^p$} spaces of entire functions in
              {${\Bbb C}^n$} and estimates of the {B}ergman kernel},
   JOURNAL = {J. Funct. Anal.},
  FJOURNAL = {Journal of Functional Analysis},
    VOLUME = {182},
      YEAR = {2001},
    NUMBER = {2},
     PAGES = {390--426},
      ISSN = {0022-1236,1096-0783},
   MRCLASS = {32A70 (32A25 46E15)},
  MRNUMBER = {1828799},
MRREVIEWER = {Joaquim\ Ortega Cerd\`a},
       DOI = {10.1006/jfan.2000.3733},
       URL = {https://doi.org/10.1006/jfan.2000.3733},
}

@article {Zelditch98,
    AUTHOR = {Zelditch, S.},
     TITLE = {Szeg{\"o} kernels and a theorem of {T}ian},
   JOURNAL = {Internat. Math. Res. Notices},
  FJOURNAL = {International Mathematics Research Notices},
      YEAR = {1998},
    NUMBER = {6},
     PAGES = {317--331},
      ISSN = {1073-7928,1687-0247},
   MRCLASS = {32L10 (47B35)},
  MRNUMBER = {1616718},
MRREVIEWER = {Thierry\ Bouche},
       DOI = {10.1155/S107379289800021X},
       URL = {https://doi.org/10.1155/S107379289800021X},
}

@incollection {Catlin,
    AUTHOR = {Catlin, D.},
     TITLE = {The {B}ergman kernel and a theorem of {T}ian},
 BOOKTITLE = {Analysis and geometry in several complex variables ({K}atata,
              1997)},
    SERIES = {Trends Math.},
     PAGES = {1--23},
 PUBLISHER = {Birkh\"auser Boston, Boston, MA},
      YEAR = {1999},
      ISBN = {0-8176-4067-3},
   MRCLASS = {32A25 (32L10 46E22)},
  MRNUMBER = {1699887},
MRREVIEWER = {Gregor\ Herbort},
}

@incollection {Zeldtich10,
    AUTHOR = {Zelditch, S.},
     TITLE = {Recent developments in mathematical quantum chaos},
 BOOKTITLE = {Current developments in mathematics, 2009},
     PAGES = {115--204},
 PUBLISHER = {Int. Press, Somerville, MA},
      YEAR = {2010},
      ISBN = {978-1-57146-146-9},
   MRCLASS = {58J51 (37D40 37N20)},
  MRNUMBER = {2757360},
MRREVIEWER = {Tatyana\ E.\ Barron},
}

@book {CoherentStates,
    AUTHOR = {Combescure, M. and Robert, D. },
     TITLE = {Coherent states and applications in mathematical physics},
    SERIES = {Theoretical and Mathematical Physics},
   EDITION = {Second},
 PUBLISHER = {Springer, Cham},
      YEAR = {[2021] \copyright 2021},
     PAGES = {xvii+576},
      ISBN = {978-3-030-70844-3; 978-3-030-70845-0},
   MRCLASS = {81R30 (81-02)},
  MRNUMBER = {4277789},
       DOI = {10.1007/978-3-030-70845-0},
       URL = {https://doi.org/10.1007/978-3-030-70845-0},
}

@article {CG03,
    AUTHOR = {Cordero, Elena and Gr\"ochenig, Karlheinz},
     TITLE = {Time-frequency analysis of localization operators},
   JOURNAL = {J. Funct. Anal.},
  FJOURNAL = {Journal of Functional Analysis},
    VOLUME = {205},
      YEAR = {2003},
    NUMBER = {1},
     PAGES = {107--131},
      ISSN = {0022-1236,1096-0783},
   MRCLASS = {47G30 (35S05 42C25 47B10 94A12)},
  MRNUMBER = {2020210},
MRREVIEWER = {Krzysztof\ Nowak},
       DOI = {10.1016/S0022-1236(03)00166-6},
       URL = {https://doi.org/10.1016/S0022-1236(03)00166-6},
}

@article {NT22,
    AUTHOR = {Nicola, Fabio and Tilli, Paolo},
     TITLE = {The {F}aber-{K}rahn inequality for the short-time {F}ourier
              transform},
   JOURNAL = {Invent. Math.},
  FJOURNAL = {Inventiones Mathematicae},
    VOLUME = {230},
      YEAR = {2022},
    NUMBER = {1},
     PAGES = {1--30},
      ISSN = {0020-9910,1432-1297},
   MRCLASS = {49Q10 (42B10 49Q20 49R05 81S30 94A12)},
  MRNUMBER = {4480144},
MRREVIEWER = {Juha\ K.\ Kinnunen},
       DOI = {10.1007/s00222-022-01119-8},
       URL = {https://doi.org/10.1007/s00222-022-01119-8},
}

@article {FeiNow,
    AUTHOR = {Feichtinger, H. G. and Nowak, K.},
     TITLE = {A {S}zeg{\"o}-type theorem for {G}abor-{T}oeplitz localization
              operators},
   JOURNAL = {Michigan Math. J.},
  FJOURNAL = {Michigan Mathematical Journal},
    VOLUME = {49},
      YEAR = {2001},
    NUMBER = {1},
     PAGES = {13--21},
      ISSN = {0026-2285,1945-2365},
   MRCLASS = {47B35 (47A10 47B06)},
  MRNUMBER = {1827072},
MRREVIEWER = {David\ F.\ Walnut},
       DOI = {10.1307/mmj/1008719032},
       URL = {https://doi.org/10.1307/mmj/1008719032},
}

@article {PLS21,
    AUTHOR = {S{\o}ndergaard, Peter L.},
     TITLE = {Efficient algorithms for the discrete {G}abor transform with a
              long {FIR} window},
   JOURNAL = {J. Fourier Anal. Appl.},
  FJOURNAL = {The Journal of Fourier Analysis and Applications},
    VOLUME = {18},
      YEAR = {2012},
    NUMBER = {3},
     PAGES = {456--470},
      ISSN = {1069-5869,1531-5851},
   MRCLASS = {65T50 (42C15 42C40 94A11)},
  MRNUMBER = {2921080},
MRREVIEWER = {Timofey\ V.\ Rodionov},
       DOI = {10.1007/s00041-011-9210-5},
       URL = {https://doi.org/10.1007/s00041-011-9210-5},
}

@incollection {BerUniversality,
    AUTHOR = {Berman, Robert J.},
     TITLE = {Determinantal point processes and fermions on polarized
              complex manifolds: bulk universality},
 BOOKTITLE = {Algebraic and analytic microlocal analysis},
    SERIES = {Springer Proc. Math. Stat.},
    VOLUME = {269},
     PAGES = {341--393},
 PUBLISHER = {Springer, Cham},
      YEAR = {2018},
      ISBN = {978-3-030-01588-6; 978-3-030-01586-2},
   MRCLASS = {32A25 (32W20 60B20 60G55)},
  MRNUMBER = {3903320},
MRREVIEWER = {Daniel\ Belti\c t\u a},
       DOI = {10.1007/978-3-030-01588-6\_5},
       URL = {https://doi.org/10.1007/978-3-030-01588-6_5},
}

@book{Watrous2018,
  author    = {Watrous, John},
  title     = {The Theory of Quantum Information},
  publisher = {Cambridge University Press},
  address   = {Cambridge},
  year      = {2018},
  doi       = {10.1017/9781316848142}
}

@article {Mirsky,
    AUTHOR = {Mirsky, L.},
     TITLE = {Symmetric gauge functions and unitarily invariant norms},
   JOURNAL = {Quart. J. Math. Oxford Ser. (2)},
  FJOURNAL = {The Quarterly Journal of Mathematics. Oxford. Second Series},
    VOLUME = {11},
      YEAR = {1960},
     PAGES = {50--59},
      ISSN = {0033-5606,1464-3847},
   MRCLASS = {15.00},
  MRNUMBER = {114821},
MRREVIEWER = {Ky\ Fan},
       DOI = {10.1093/qmath/11.1.50},
       URL = {https://doi.org/10.1093/qmath/11.1.50},
}
\bibliographystyle{amsalpha}

\end{document}